\documentclass[12pt]{amsart}
\usepackage{geometry}
\usepackage{amsmath, amssymb}
\usepackage{amscd}
\usepackage{enumitem}
\usepackage{comment}
\usepackage[english]{babel}
\usepackage{graphicx}
\usepackage{epstopdf}
\numberwithin{equation}{section}

\newtheorem{thm}{Theorem}[section]
\newtheorem{lem}[thm]{Lemma}
\newtheorem{prop}[thm]{Proposition}

\theoremstyle{definition}

\theoremstyle{remark}
\newtheorem{rem}[thm]{Remark}

\begin{document}

\centerline{\bf Automorphisms of the Baumslag-Gersten group.}

\centerline{\bf Boris Lishak}
\centerline{\bf Department of Mathematics, University of Toronto}
\centerline{\bf 40 St. George st., Toronto, ON, M5S 2E4, CANADA;}
\centerline{\bf bors@math.toronto.edu}
\vskip 1truecm

{\bf Abstract.} We classify homomorphisms of the Baumslag-Gersten group into itself. We prove it is Hopfian and co-Hopfian. We show that the group of outer automorphisms of the Baumslag-Gersten group is isomorphic to the dyadic rationals with the addition operation. These results are not new. They were obtained by Andrew M. Brunner in \cite{Brunner}. However, our exposition is self-contained and, hopefully, more accessible for some readers. 
\section{Introduction} \label{intro}

The Baumslag-Gersten group introduced in \cite{Baum} enjoys the interesting property of having the fast growing Dehn function (see \cite{Gers1}, \cite{Gers2}, \cite{Plat}), while it can be presented by the number of relators one less than the number of generators. This property was found to be useful in constructing Riemannian metrics which are non-trivial minimums of some functionals of Riemannian structures on $S^n$ (and other $4$-dimensional manifolds). (see \cite{LN}). We hope that the following description of the outer automorphisms of the Baumslag-Gersten group will help us to find exponentially growing number of such metrics and further our understanding of the spaces of metrics on $4$-dimensional manifolds.
\\

We can present the Baumslag-Gersten group by $G = <x,y, t| x^y = x^2, x^t = y >$, where $a^b$ denotes $b^{-1} a b$. From this presentation we can see that this group is an HNN extension of the Baumslag-Solitar group, presented by $H = <x,y| x^y = x^2>$. The latter is an HNN extension of the group $<x>$.  We state our main result.

\begin {thm} \label{main}
$Out(G)$ is isomorphic to the dyadic rationals - $<\frac{1}{2}, \frac{1}{4}, \frac{1}{8}, ...>$, the subgroup of the rationals with addition, which is not finitely generated.
\end {thm}

To prove this theorem we classify all homomorphisms of $G$ into itself, then prove $Out(G)$ is isomorphic to $C_H(x)$, the centralizer of $x$ in $H$, which we then describe as a subgroup of rational numbers. We use the techniques of Collins \cite{Collins}, mainly the Collins' Lemma (\cite{Collins}, \cite{LS}, \cite{Short}). We present this tool in the next section and prove the main results in the last. These and similar results for a class of groups that includes $G$ were obtained in \cite{Brunner}.

\section{HNN extensions and Collins' lemma} \label{HNN}

Recall, we can represent elements of HNN-extensions by reduced sequences of the stable letter, its inverse and elements of the base group . We denote by $|w|$ the number of occurences of the stable letter and its inverse in the reduced form of w. We can also cyclically permute these sequences. If all the premutations are also reduced we call this element cyclically reduced. Any element is conjugate to a cyclically reduced element. Here is a form of Collins lemma we will use (see also \cite{LS}, \cite{Short}). The statement is cumbersome but the proof is simple.

\begin {lem} \label{collins}
Let $u$,$v$ be cyclically reduced conjugate elements in some HNN-extension (given by $t^{-1}Ct = B$, where $C$,$B$ are subgroups of the base group). $u=g^{-1} v g$, where $g = a_1 t^{\sigma_1} a_2 t^{\sigma_2} ... a_n t^{\sigma_n}$. Then $|u|=|v|$. Furhermore, if $|u|=|v| = 0$, then there is a finite chain of words $v_1$,$u_1$...$v_n$, $u_n$, where $u_i, v_i \in C \cup B$ such that either $u_i = t^{-1}v_i t$ or $v_i = t^{-1}u_i t$ and $v$ is conjugate to $u_1$, $u$ is conjugate to $v_n$, $v_i$ is conjugate to $u_{i+1}$, where all conjugations are in the base group (in fact by $a_i$).

\end{lem}

\begin {proof}
Consider a van Campen's diagram for $u=g^{-1} v g$. Note that $t$-bands on the diagram can not start and end on the same $u$ or $v$ because they are cyclically reduced. Therefore there are two possibilities for the way the $t$-bands go on the diagram (Figure \ref{figcollins}) corresponding to $|u|=|v| = 0$ and $|u|=|v| \neq 0$. The second part of the statement follows from Figure \ref{figcollins} too.
\end {proof}

\begin{figure}[h]
\centering
\includegraphics[scale=0.5]{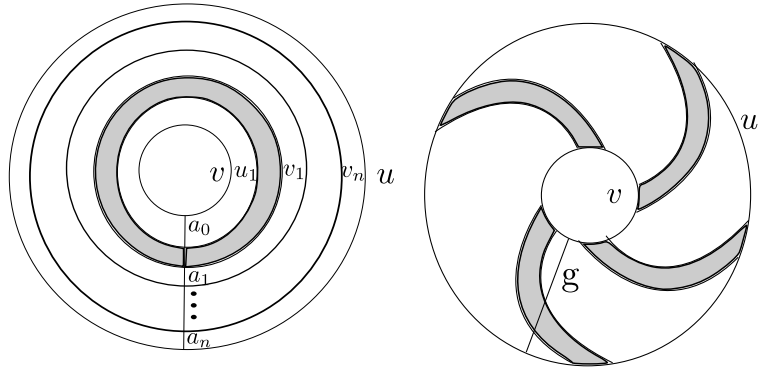}
\caption{Two possible Van Kampen diagrams for $u=g^{-1} v g$. The inside circle is $v$, the outside is $u$, the vertical line is $g$. $t$-bands are marked by grey.}
\label{figcollins}
\end{figure}

\section{Main Results} \label{results}

We prove several lemmas that we will apply in order to a homomorphism $G \to G$.

\begin {lem} \label{lem1}
Let $F: G \to G$ be a homomorphism, then there exists an inner automorphism $A$ such that $(A \circ F)(x)$ is in $<x,y>$.
\end{lem}

\begin {proof}
We can choose $A$ such that $(A\circ F)(x)$ (and therefore $((A\circ F)(x))^2$) is cyclically reduced, then it follows from Lemma \ref{collins} applied to $(A\circ F)(x) =  g^{-1} ((A\circ F)(x))^2 g$ that $|(A\circ F)(x)|=0$.
\end {proof}

\begin {lem} \label{lem2}
Let $F: G \to G$ be a homomorphism, such that $F(x) \in <x,y>$, then there exists an inner automorphism $A$ such that $(A \circ F)(x) = x^i$.
\end{lem}

\begin {proof}
Consider $(F(y))^{-1}F(x)F(y)=(F(x))^2$ and apply Collins' lemma to it (if there are $t$ letters in $F(y)$ apply it to the second HNN extension (stable letter $t$), otherwise apply to the first one (stable letter $y$). In the latter case F(x) might not be cyclically $y$ reduced, but its conjugate is. In either case $v_1$ or $u_1$ will be $x^{i}$ (see Figure \ref{figlem2}). In the degenerate case of $F(y)=1$ we have $F(x) = 1 = x^{0}$.
\end {proof}

\begin{figure}[h]
\centering
\includegraphics[scale=0.6]{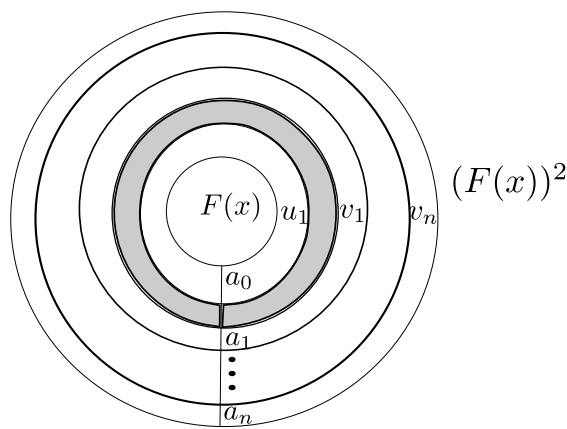}
\caption{A Van Kampen diagram for $(F(y))^{-1}F(x)F(y)=(F(x))^2$. The inside circle is $F(x)$, the outside is $(F(x))^2$, the vertical line is $F(y)$. $t$-bands are marked by grey. One of the edges of the $t$-band is $x^i$.}
\label{figlem2}
\end{figure}

\begin {lem} \label{lem3}
Let $F: G \to G$ be a homomorphism, such that $F(x) = x^i \neq 1$, then $F(y) \in <x,y>$.
\end{lem}

\begin {proof}
Consider $(F(y))^{-1}x^iF(y)=x^{2i}$ and apply Collins' lemma to it. The outside of the first $t$-ring is $y^j$ and it has to be Baumslag-Solitar conjugate to the inside of the second $t$-ring (see Figure \ref{figlem3}), which is either $x^n$ or $y^m$. The first case is impossible, and the second case is only possible if $y^j$ is conjugated by $y^k$ (in the Baumslag-Solitar group anything else conjugates $y^j$ outside of $<y>$). Then the first and the second $t$-rings can be reduced. There could not be only one $t$-ring because the outside of the diagram is $x^{2i}$. Therefore $F(y)$ has no letters $t$ (in reduced form).

\end {proof}

\begin{figure}[h]
\centering
\includegraphics[scale=0.5]{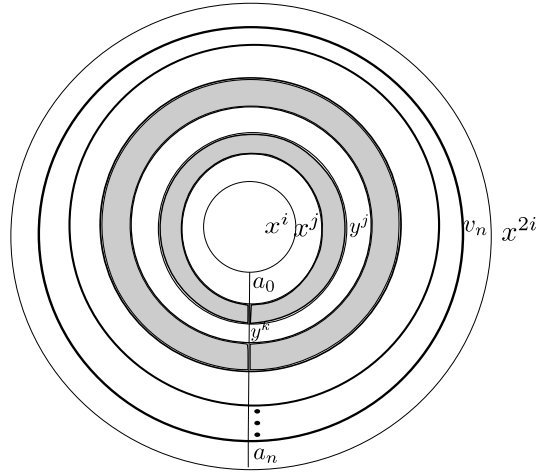}
\caption{A Van Kampen diagram for $(F(y))^{-1}F(x)F(y)=(F(x))^2$. The inside circle is $F(x)$, the outside is $(F(x))^2$, the vertical line is $F(y)$. $t$-bands are marked by grey. One of the edges of the $t$-band is $x^i$.}
\label{figlem3}
\end{figure}

\begin {lem} \label{lem4}
Let $F: G \to G$ be a homomorphism, such that $F(x) = x^i \neq 1$ and $F(y) \in <x,y>$. Then $F(t)$ has one letter $t$, and there is an inner automorphism $A$ such that $(A \circ F)(x)=x$ and $(A \circ F)(y)=y$.
\end{lem}

\begin {proof}
$F(y) = y^{n+1}x^{j}y^{-n}$ for $n>0$ because of the equality $(F(y))^{-1}x^iF(y)=x^{2i}$. If $F(t)$ had no $t$ letters we would have that $x^i$ is conjugate to $y^{n+1}x^{j}y^{-n}$ in Baumslag-Solitar, which is impossible: the sum of exponents of $y$ can not be changed by conjugation. As in the proof of Lemma \ref{lem3} $F(t)$ can have at most one letter $t$ (in reduced form). Let $F(t) = g_1 t g_2$, where $g_1,g_2$ are in the Bauslag-Solitar. Again by Collins' lemma ${g_1}^{-1} x^i g_1 = x^m$ and ${g_2}^{-1} y^m g_2 = F(y)$. The left hand side of the last equation has sum of exponents of $y$ equal to $1$, therefore $m=1$, which implies there exists an inner automorphism $A'$ such that $(A' \circ F)(x)=x^m = x$. Let us redefine $n,j$ such that $(A' \circ F) (y)  = y^{n+1}x^{j}y^{-n}$.

 Let $A''$ be the conjugation by $y^n x^j y ^ {-n}$, then $(y^{n} x^{-j} y ^ {-n}) x (y^n x^j y ^ {-n}) = x$ and \\
$(y^{n} x^{-j} y ^ {-n}) (y^{n+1}x^{j}y^{-n}) (y^n x^j y ^ {-n}) = (y^{n} x^{-j}) y x^{j} ( x^j y ^ {-n}) = y^{n} y y ^ {-n} = y$. We get the conclusion of the lemma by setting $A = A'' \circ A'$.
\end {proof}

\begin {prop} \label{homo}
Let $F: G \to G$ be a homomorphism. If $F(x)=1$, then $F(y)=1$ and $F(t)$ can be any element. Otherwise, $F$ is an automorphism, and there exists an inner automorphism $A$ such that $(A \circ F)(x) = x$, $(A \circ F)(y) = y$, and $(A \circ F)(t) = gt$, where $g \in C_H(x)$.

\end {prop}

\begin {proof}
If $F(x) \neq 1$ we apply Lemmas \ref{lem1}, \ref{lem2}, \ref{lem3}, \ref{lem4} in succession to obtain $A$, such that $(A \circ F)(x) = x$, $(A \circ F)(y) = y$, and $(A \circ F)(t) = g_1 t g_2$, where $g_1, g_2$ are in Baumslag-Solitar. Since ${g_2}^{-1} y^k g_2 = y$, we have $g_2 = y^m$ and $k=1$. Therefore ${g_1}^{-1} x^{k} g_1 = {g_1}^{-1} x g_1 = x$, i.e. $g_1 \in C_H(x)$. Since $ g_1 t g_2 =  g_1 t y^m = g_1 x^m t$, we set $g=g_1 x^m \in C_H(x)$.

\end {proof}

\begin {rem} \label{hopf}

It follows that $G$ is both Hopfian and co-Hopfian. Now we can describe $Out(G)$.
\end {rem}

\begin {prop} \label{centre}
$Out(G)$ is isomorphic to $C_H(x)$, the centralizer of $x$ in the Baumslag-Solitar group. The isomorphism $\phi : C_H(x) \to Out(G)$ is defined by $\phi(g) = [F]$, where $F : x \mapsto x$, $y \mapsto y$, $t \mapsto gt$ is an automorphism.
\end {prop}

\begin {proof}
This map is one-to-one because $C_H(<x,y>) = {1}$. This map is onto because of Proposition \ref{homo}. The multiplication is clearly preserved.
\end {proof}

\begin {rem} \label{last}
$C_H(x) = \{ y^{n}x^{j}y^{-n} \}$, all elements of $H$ that have total power of $y$ being $0$. This subgroup is generated by $y^{i} x y^{-i}$, $i>0$. We can construct a map $\psi : C_H(x) \to \mathbf{Q}$, defined by  $y^{i} x y^{-i} \mapsto \frac{1}{2^{i}}$. This map is an injective homomorphism with respect to addition: $\psi(y^{n}x^{j}y^{-n}) = \frac{j}{2^{n}}$. The image of this map is not finitely generated, because all of its finitely generated subgroups are cyclic, but the image itself is not cyclic.
\end {rem}

The main theorem, Theorem \ref{main}, follows from Proposition \ref{centre} and Remark \ref{last}.

\end{document}